\documentclass[article, 11pt]{amsart}

\usepackage{amsmath, amssymb, amsfonts, amsthm, latexsym}
\usepackage{tikz}

\newtheorem{Theorem}{Theorem}[section]
\newtheorem{Lemma}[Theorem]{Lemma}
\newtheorem{Corollary}[Theorem]{Corollary}
\newtheorem{Proposition}[Theorem]{Proposition}
\newtheorem{Remark}[Theorem]{Remark}
\newtheorem{Example}[Theorem]{Example}

\def\reg{\operatorname{reg}}

\def\ini{\operatorname{ini}}
\def\depth{\operatorname{depth}}

\def\ini{{\operatorname{in}}}
\def\Exp{\operatorname{Exp}}

\def\mm{{\mathfrak m}}
\def\nn{{\mathfrak n}}
\def\q{{\mathfrak q}}

\def\NN{{\mathbb N}}
\def\a{{\mathbf a}}

\def\e{{\mathbf e}}

\begin{document}

\title{Castelnuovo-Mumford regularity\\ and Ratliff-Rush closure}

\author{Maria Evelina Rossi}
\address{Department of Mathematics, University of Genoa, Via Dodecaneso 35, 16146 Genoa, Italy}
\email{rossim@dima.unige.it}

\author{Dinh Thanh Trung}
\address{Department of Mathematics, FPT University, 8 Ton That Thuyet, My Dinh, Tu Liem, Hanoi, Vietnam}
\email{trung.dinh.nb@gmail.com}

\author{Ngo Viet Trung}
\address{Institute of Mathematics, Vietnam Academy of Science and Technology, 18 Hoang Quoc Viet, Hanoi, Vietnam}
\email{nvtrung@math.ac.vn}

\thanks{ The two last authors are supported by Vietnam National Foundation for Science and Technology Development under grant number 101.04-2017.19. Part of the paper was carried out when the last author visited Genova in July 2014. 
He would like to thank INdAM (GNSAGA) and the Department of Mathematics of the University of Genova for their  support and hospitality.} 


\begin{abstract}
We establish strong relationships between the Castelnuovo-Mumford regularity  and the Ratliff-Rush closure of an ideal. 
Our results have several interesting consequences on the computation of the Ratliff-Rush closure, the stability of the Ratliff-Rush filtration, the invariance of the reduction number, and the computation of the Castelnuovo-Mumford regularity of the Rees algebra and the fiber ring. In particular, we prove that the Castelnuovo-Mumford regularity of the Rees algebra and of the fiber ring are equal for large classes of monomial ideals in two variables, thereby verifying a conjecture
of Eisenbud and Ulrich for these cases.
\end{abstract}

\maketitle

\centerline{\small \it Dedicated to Giuseppe Valla on the occasion of his seventieth anniversary} \vskip 0.7cm

\section*{Introduction}

The Castelnuovo-Mumford regularity is a homological invariant which controls the complexity of graded structures over a polynomial rings \cite{EG}. It is well known that the Castelnuovo-Mumford regularity plays an important role in computational Algebraic Geometry and Commutative Algebra \cite{BM, Va}. It is less known that the Castelnuovo-Mumford regularity can be also defined for graded structures over a commutative ring and that the Castelnuovo-Mumford regularity of the Rees algebra controls the complexity of an ideal; see e.g. \cite{Tr1,Tr2,Tr3}. The main aim of this paper is to explore this aspect of the Castelnuovo-Mumford regularity by establishing relationships between the Castelnuovo-Mumford regularity of the Rees algebra and the seemingly unrelated Ratliff-Rush closure of an ideal. 
\par

The motivation for our work originates from the following conjecture of Eisenbud and Ulrich \cite[Conjecture 1.3]{EU}. \medskip

\noindent {\bf Conjecture}.  
Let $A$ be a standard graded algebra over a field $k$.
Let $I$ be an $\mm$-primary graded ideal generated by forms of the same degree, where $\mm$ denotes the maximal graded ideal of $A$. Then   
$\reg R(I) = \reg F(I),$ 
where $R(I) = \oplus_{n \ge 0}I^n$ is the Rees algebra and $F(I) = \oplus_{n\ge 0}I^n/\mm I^n$ is the fiber ring of $I$. \medskip
   
For an arbitrary graded ideal $I$, it is known that $\reg I^n$ is asymptotically a linear function \cite{CHT, Ko, TW}. 
However, very little is known on the stability index of $\reg I^n$, i.e. the least number $n$ where $\reg I^n$  becomes a linear function afterward. Eisenbud and Ulrich \cite{EU} showed that under the above assumption, $\reg I^n$ is related to the presentation of $R(I)$ as a direct sum of modules over $F(I)$, which has led them to raise the above conjecture.  \par

The conjecture of Eisenbud and Ulrich is not true if one does not put further assumption on the base ring $A$. We shall see that if the conjecture were true, then $A$ must be a Buchsbaum ring. In fact, Ulrich communicated to the authors that the conjecture should be formulated  for polynomial rings. To give an answer to 
 this modified conjecture seems to be difficult because there were no tools, which allow us to compare $\reg R(I)$ and $\reg F(I)$. 
Note that $R(I)$ also has an $\NN$-graded structure over $k$, which has a different Castelnuovo-Mumford regularity. 
This Castelnuovo-Mumford regularity was studied by Herzog, Popescu and Trung in \cite{HPT}. \par 

We shall see in this paper that superficial sequences of $I$ can be used as a common tool to characterize 
 $\reg R(I)$ and $\reg F(I)$. As a first consequence, we show that $\reg R(I) = \reg F(I)$ provided $ \depth G(I) \ge \dim A-1$, where $G(I) := \oplus_{n \ge 0} I^n/I^{n+1}$ is the associated graded ring of $I$. 
Our main results are however the findings that $\reg R(I)$ and $\reg F(I)$ control the behavior of  the Ratliff-Rush closure and the Ratliff-Rush filtration of $I$. If $\dim A = 2$, we can express $\reg R(I)$ and $\reg F(I)$ in terms of the stability of the Ratliff-Rush filtration. These relationships are of independent interests as we shall see below. \par

Let $(A,\mm)$ be an arbitrary local ring and $I$ an arbitrary ideal of $A$. 
The Ratliff-Rush closure of $I$ is defined as the ideal
$\tilde I := \bigcup_{n \ge 1} I^{n+1}:I^n.$
It is a refinement of the integral closure of $I$ and $\tilde I = I$ if $I$ is integrally closed. If $I$ is a regular ideal,  i.e. if $I$ contains non-zerodivisors, $\tilde I$ is the largest ideal sharing the same higher powers with $I$ \cite{RR}. In particular, the Ratliff-Rush filtration $\{\widetilde {I^n}\}$, carries important information on the blowups, the associated graded ring of $I$, and the Hilbert function of an $\mm$-primary ideal $I$ (see e.g. \cite{HJLS,HLS,Ho,Sa}). \par

In general, the computation of $\tilde I$ is hard because $I^{n+1}:I^n = I^n:I^{n-1}$ does not imply $I^{n+2}:I^{n+1} = I^{n+1}:I^n$ \cite{RS}.  Let  $s(I)$ denote
the least integer $m \ge 0$ such that $\tilde I = I^{n+1}:I^n$ for all $n \ge m$.   
We show that 
$$s(I) \le \max\{\reg R(I)-1,0\}.$$
Since there are various bound for $\reg R(I)$ in terms of other well known invariants of $I$ (\cite{DGV}, \cite{DH}, \cite{Li1}, \cite{Li2}, \cite{RTV1}, \cite{St},  \cite{Va}), 
the above bound provides us an {\em effective tool for the computation of $\tilde I$} because $\tilde I = I^{c+1}:I^c$ if $c$ is an upper bound for $\reg R(I)$. As far as we know, an upper bound for $s(I)$ has been given before only for $\mm$-primary ideals in Cohen-Macaulay rings by Elias \cite{El}.  \par

A remarkable feature of the Ratliff-Rush filtration is that $\widetilde {I^n} = I^n$ for  $n \gg 0$ if $I$ is a regular ideal \cite{RR}. Again, if $\widetilde {I^n} = I^n$, it does not necessarily imply $\widetilde {I^{n+1}} = I^{n+1}$ \cite{RS}. Let $s^*(I)$ denote the least integer $m \ge 1$ such that $\widetilde {I^n} = I^n$ for all $n \ge m$.   This invariant was studied to some extent in \cite{Pu, RS}.
We show that 
$$s^*(I) \le \max\{\reg R(I),1\}.$$

On the other hand, we can characterize $\reg R(I)$ in terms of $s^*(I)$ if $I$ is an $\mm$-primary ideal, which is not a parameter ideal, in a two-dimensional regular local ring. In fact, we show that
$$\reg R(I) = \max\{r_J(I), s^*(I)\}  = \min\{n \ge r_J(I)|\ \widetilde {I^n} = I^n\},$$
where $J$ is an arbitrary minimal reduction of $I$ and $r_J(I)$ denotes the reduction number of $I$ with respect to $J$. 
The first formula has interesting consequences on the invariance of the reduction number, while the last formula gives {\em an effective tool for the computation of $\reg R(I)$}.  Rossi and Swanson in \cite[Section 4]{RS} asked whether $s^*(I) \le r_J(I)$ holds in general. Using Huckaba's example on the non-invariance of the reduction number \cite{Huc} we give a negative answer to this question. \par

Similarly, we show that if $A$ is a polynomial ring in two variables and $I$ is a homogeneous $\mm$-primary ideal generated by forms of the same degree $d$, which is not a parameter ideal, then  
$$\reg F(I) =   \min\{n \ge r_J(I)|\ (\widetilde {I^n})_{nd} = (I^n)_{nd}\}$$
for any homogeneous minimal reduction $J$ of $I$. 
Since $nd$ is the initial degree of both ideals $I^n$ and $\widetilde {I^n}$, this formula shows that $\reg F(I)$ controls the difference between $\widetilde {I^n}$ and $I^n$ at the initial degree, while $\reg R(I)$ controls the difference between $\widetilde {I^n}$ and $I^n$ at all degrees. \par

The above formulae can be used effectively to compare $\reg R(I)$ and $\reg F(I)$.  In fact,
we will use them to prove that $\reg R(I) = \reg F(I)$ for large classes of monomial ideals in two variables.
To compute the Ratliff-Rush closure of monomial ideals generated by monomials of the same degree in two variables one may use a method developed by Crispin Quinonez in \cite{Cr}. \par

We note that Cortadellas and Zarzuela \cite{CZ}, and Jayanthan and Nanduri \cite{JN} gave some conditions for the equality $\reg R(I) = \reg F(I)$ in the case $I$ is an ideal in a local ring. However, their results are too specific to be recalled here. \par

The paper is divided into four sections. 
In Section 1 we present a characterization of $\reg R(I)$ in terms of a superficial sequence of $I$ and discuss its relationship to $s(I)$. In Section 2 we show that $\reg R(I)$ is an upper bound for $s^*(I)$ and that one can characterize $\reg R(I)$ in terms of $s^*(I)$ in the two-dimensional case. In Section 3 we give a characterization of $\reg F(I)$ in terms of a superficial sequence of $I$ and of the Ratliff-Rush filtration.  
In Section 4 we apply our results to monomial ideals in two variables.


\section{Regularity of the Rees algebra}

Let $(A,\mm)$ be a local ring with $\dim A > 0$ and $I$ an ideal of $A$. Without loss of generality we may assume that the residue field of $A$ is infinite.\par

An element $x \in I$ is called {\em superficial} for $I$ if
there is an integer $c$ such that 
$$(I^{n+1} : x) \cap I^c = I^n$$
for all large $n$. 
A system of elements  $x_1,....,x_s$ in $I$ is called a {\em superficial sequence} of $I$ if $x_i$ is a superficial element
of $I$ in $A/(x_1,....,x_{i-1})$, $i = 1,...,s$. A superficial sequences can be constructed by means of the following notion. \par

A system of homogeneous elements $z_1,...,z_s$ in the associated graded ring $G(I)$ is called {\em filter-regular} if 
$$[(z_1,...,z_{i-1}):z_i]_n = (z_1,...,z_{i-1})_n,$$
for sufficiently large $n$, $i = 1,...,s$. It is easy to see that $z_1,...,z_s$ is filter-regular if and only if $x_i \not\in P$ for all associated primes $P \not\supseteq G(I)_+$ of $(z_1,...,z_{i-1})$, $i = 1,...,s$ (see \cite{Tr1}). \par

For every element $x \in I$ we denote by $x^*$ the  residue class of $x$ in $I/I^2 \subset G(I)$.  

\begin{Lemma} \label{superficial}  \cite[Lemma 6.2]{Tr2}
$x_1,....,x_s$ is a superficial sequence of $I$ if and only if $x_1^*,...,x_s^*$ forms a filter-regular sequence of $G(I)$.
\end{Lemma}

An ideal $J \subseteq I$ is called a {\em reduction} of $I$ if there exists an integer $n$ such that $I^{n+1} = JI^n$.
The least integer $n$ with this property is called the {\em reduction number} of $I$ with respect to $J$.
We will denote it by $r_J(I)$. A reduction is minimal if it is minimal with respect to containment.

\begin{Lemma} \label{generating}  
Every minimal reduction $J$ of $I$ can be generated by a superficial sequence of $I$.
\end{Lemma}

\begin{proof}
Let $Q$ denote the ideal in $G(I)$ generated by the elements $x^*$, $x \in J$. Then $Q$ is generated by $(J+I^2)/I^2$.
Since $I^{n+1} = JI^n$,  $Q_{n+1} = G(I)_{n+1}$. Therefore, $Q \not\subseteq P$ for any prime $P \not\supseteq G(I)_+$. Using prime  avoidance we can find a filter-regular sequence $z_1,...,z_s \in G(I)$ such that $Q = (z_1,...,z_s)$. 
Choose $x_i \in J$ such that $x_i^* = z_i$, $i = 1,...,s$.
Then $(x_1,...,x_s) + I^2 = J + I^2$. Hence 
$$(x_1,...,x_s)I^n + I^{n+2} = JI^n + I^{n+2} = I^{n+1}.$$
By Nakayama's Lemma, this implies $I^{n+1} = (x_1,...,x_s)I^n$. 
Therefore, $(x_1,...,x_s)$ is a reduction of $I$. By the minimality of $J$, we must have $J =(x_1,...,x_s)$. The conclusion now follows from Lemma \ref{superficial}.
\end{proof} 

The above notions can be used to characterize the Castelnuovo-Mumford regularity of $R(I)$. 
Recall that the Castelnuovo-Mumford regularity of a standard graded algebra $R$ over a commutative ring is defined as follows. 
Let $H_{R_+}^i(R)$ denote the $i$-th local cohomology module of $R$ with respect to the graded ideal $R_+$ of elements of positive degree and set $a_i(R) =  \max\{n|\ H_{R_+}^i(R)_n \neq 0\}$ with the convention $a_i(R) = -\infty$ if $H_{R_+}^i(R) = 0$. Then one defines
$$\reg R := \max\{a_i(R)+i|\ i \ge 0\}.$$ 
\par

It is known that $\reg R(I)$ can be characterized in terms of a reduction $J = (x_1,...,x_s)$ of $I$ such that 
$x_1^*,...,x_s^*$ form a filter-regular sequence in $G(I)$ \cite[Theorem 4.8]{Tr2}.  
By Lemma \ref{superficial} and Lemma \ref{generating} we can replace a filter-regular sequence in $G(I)$ by a superficial sequence of $I$ as follows.

\begin{Theorem} \label{regularity}  
Let $x_1,...,x_s$ be a superficial sequence of $I$ such that $J = (x_1,...,x_s)$ is a reduction of $I$. Then 
\begin{align*} & \reg R(I) = \reg G(I)\\
& = \min\left\{n \ge r_J(I)|\  I^{n+1} \cap [(x_1,...,x_{i-1}):x_i] = (x_1,...,x_{i-1})I^n, i = 1,...,s\right\}.
\end{align*}
Moreover, the equations in the above formula holds for all $n \ge \reg R(I)$.
\end{Theorem}

Theorem \ref{regularity} will play a crucial role in our paper. 
It can be also used to compute $\reg F(I)$ as we shall see later. 

\begin{Corollary}  \label{intersection} 	
Let $x$ be a superficial element of $I$. Then
$I^{n+1} \cap (x) = xI^n$ for $n \ge \reg R(I)$.
\end{Corollary}

\begin{proof}  
Using Lemma \ref{superficial} we can extend $x_1 := x$ to a superficial sequence $x_1,...,x_s$ of $I$ such that $J = (x_1,...,x_s)$ is a reduction of $I$, $s \ge 2$. By Theorem \ref{regularity}, we have 
$I^{n+1} \cap [(x_1):x_2] = x_1I^n$ for $n \ge \reg R(I)$.
Since 
$$I^{n+1} \cap [(x_1):x_2] \supseteq I^{n+1} \cap  (x_1) \supseteq  x_1I^n,$$
we obtain $I^{n+1} \cap (x_1) = x_1I^n$ for $n \ge \reg R(I)$. 
\end{proof}

Corollary \ref{intersection}  leads us to the following property of colon ideals.

\begin{Proposition} \label{colon} 	
Let $I$ be a regular ideal. Then 
$I^{n+1} : I = I^n$ for $n \ge \reg R(I)$.
\end{Proposition}

\begin{proof}
It is well-known that if $I$ is a regular ideal, every superficial element of $I$ is a non-zerodivisor (see e.g. \cite[Lemma 1.2]{RV}). Therefore, $0:x = 0$ if $x$ is a superficial element for $I$.  
By Corollary \ref{intersection}, $I^{n+1}:x = I^n + (0:x) = I^n$ for $n \ge \reg R(I)$. 
Since $I^n \subseteq I^{n+1}:I \subseteq I^{n+1}:x$, this implies $I^{n+1}:I = I^n$ for $n \ge \reg R(I)$.
\end{proof}

\begin{Remark} \label{a}
{\rm Corollary \ref{intersection} can be also deduced from \cite[Lemma 4.4 (i)]{Tr2}, whose proof shows that $I^{n+1} \cap (x) = xI^n$ for $n \ge \max\{a_0(G(I)),a_1(G(I))+1\}$. If $I$ is a regular ideal, $a_0(G(I)) < a_1(G(I))$ by \cite[Theorem 5.2]{Ho1}. Therefore, $I^{n+1} : I = I^n$ for $n \ge a_1(G(I))+1$.}
\end{Remark}

Recall that the Ratliff-Rush closure of $I$ is defined by
$$\tilde I := \bigcup_{n \ge 1} I^{n+1}:I^n.$$
Let $s(I)$ denotes the least integer $m \ge 0$ such that $I^{n+1}:I^n = \tilde I$ for all $n \ge m$.  

\begin{Theorem} \label{stable}
Let $I$ be a regular ideal. Then $s(I) \le \max\{\reg R(I)-1,0\}$.
\end{Theorem}

\begin{proof}
Applying  Proposition \ref{colon} we have 
$$I^{n+1} : I^n = (I^{n+1}:I) : I^{n-1} = I^n:I^{n-1}$$
for $n\ge \reg R(I)$. Thus, $I^{n+1}:I^n$ is the same ideal for all $n \ge \reg R(I)-1$ and equals $\tilde I$.
\end{proof}

There are plenty examples with $s(I) = 0$ and $\reg R(I)$ arbitrarily large. 
For instance, if $I = \mm$ then $s(\mm) = 0$. Since $\reg R(\mm) \ge r_J(\mm)$ for any minimal reduction $J$ of $\mm$, one can easily construct local rings such that $\reg R(\mm)$ is arbitrarily large. \par

By Theorem \ref{stable}, if we know an upper bound $c$ of
$\reg R(I)$, we can easily compute $\tilde I$ because $\tilde I = I^{c+1}:I^c$. 
There are several upper bounds for $\reg R(I)$ in terms of other invariants \cite{DGV, DH, Li1, Li2, RTV1,St,Va}. 

\begin{Corollary}  \label{extended}
Let $I$ be an $\mm$-primary ideal in a Cohen-Macaulay ring $A$ and $d = \dim A$. 
Let $e(I)$ be the the multiplicity $e(I)$ of $I$. Then \par
{\rm (i)}  $s(I) \le e(I)-1$ if $d = 1$,\par
{\rm (ii)}  $s(I) \le  e(I)^{2(d-1)!-1}[e(I)-1]^{(d-1)!}$ if $d \ge 2$.
\end{Corollary}

\begin{proof}
It was proved in \cite[Corollary 3.4]{RTV1} that the right-hand sides of the bounds in (i) and (ii) are upper bounds for  $\reg R(I)$ in the case $I = \mm$.  However, it can be checked that the proof also holds for an arbitrary $\mm$-primary ideal $I$ (cf. \cite[Corollary 4.4]{Li1}, where a more compact but weaker bound for $\reg R(I)$ is given). Therefore, the assertion follows from Theorem \ref{stable}.
\end{proof}

A similar upper bound for $s(I)$ was already given for this case by Elias \cite[Theorem 2.1(ii)]{El}, which is worse than Corollary \ref{extended}. This bounds are deduced from a bound for $s(I)$ in terms of the postulation numbers of $I$ and of ideals of the form $I/(x)$, where $x$ is an element of a given superficial sequence of $I$ generating a minimal reduction of $I$ \cite[Theorem 2.1(i)]{El}.  


\section{Ratliff-Rush filtration}

Let $(A,\mm)$ be a local ring with $\dim A > 0$ and $I$ an ideal of $A$.
One calls the sequence of ideals $\widetilde {I^n}$, $n \ge 1$, the {\em Ratliff-Rush filtration} with respect to $I$.
It is well known that for $n \ge 1$,
$$\widetilde {I^n} = \bigcup_{t \ge 0}I^{n+t}:I^t.$$

To compute $\widetilde {I^n}$, we need to know an upper bound for the least number $t$ with $\widetilde {I^n} = I^{n+t}:I^t$. The following result gives an upper bound for such $t$ in terms of $\reg R(I)$. Moreover, we also obtain an upper bound for the stability index $s^*(I)$, which is the least integer $m \ge 1$ such that $\widetilde {I^n} = I^n$ for $n \ge m$.  

\begin{Proposition} \label{bound}
Let $I$ be a regular ideal. Then \par
{\rm (i)} $\widetilde {I^n} = I^{n+t}:I^t$ for $t \ge \reg R(I)-n$,\par
{\rm (ii)} $s^*(I) \le \max\{\reg R(I),1\}$.
\end{Proposition}

\begin{proof}
By Proposition \ref{colon} we have $I^{n+1}:I = I^n$ for $n \ge \reg R(I)$.
Therefore, 
$$I^{n+t+1}: I^{t+1} = (I^{n+t+1}:I):I^t = I^{n+t}: I^t$$
for $t \ge \reg R(I)-n$, which proves (i). 
If $n \ge \reg R(I)$, we can put $t = 0$ in (i). Hence 
$\widetilde {I^n} = I^n:I^0 = I^n$, which proves (ii). 
\end{proof}

As pointed out in Remark \ref{a}, we can replace $\reg R(I)$ by $a_1(G(I))+1$ in Proposition \ref{bound}. 
So we can recover the bound  $s^*(I) \le \max\{a_1(G(I))+1,1\}$ given by Puthenpurakal in \cite[Theorem 4.3]{Pu}. Note that Puthenpurakal considers the least integer $m \ge 0$ such that $\widetilde {I^n} = I^n$ for $n \ge m$,  whereas we require $m \ge 1$ because one always has $\widetilde {I^0} = I^0 = A$. Proposition \ref{bound}(ii) can be also deduced from this bound because $a_1(G(I))+1 \le \reg G(I) = \reg R(I)$ by Theorem \ref{regularity}.  \par

Recall that a system of elements $x_1,...,x_r$ in $A$  is a {\em $d$-sequence} if the following two conditions are satisfied:\par
(i) $x_i$ is not contained in the ideal generated by the rest of the system, $i = 1,...,r$,\par
(ii) $(x_1,...,x_i):x_{i+1}x_k = (x_1,...,x_i):x_{i+1}$ for all $i = 0,...,r-1$ and $k = i+1,...,r$. \par
\noindent This notion was introduced by Huneke in \cite{Hun}. 
Examples of $d$-sequences are abundant such as the maximal minors of an $r  \times (r+ 1)$ generic matrix and systems of parameters in Buchsbaum rings. \par

It was known that $I$ is generated by a $d$-sequence if and only if $\reg R(I) = 0$ \cite[Corollary 5.7]{Tr2}. By Proposition \ref{bound}(ii), $\reg R(I) = 0$ implies $s^*(I) = 1$. This means $\widetilde {I^n} = I^n$ for all $n \ge 1$. So we obtain the following consequence.

\begin{Corollary} \label{d-sequence}
Let $I$ be a regular ideal generated by a $d$-sequence. Then $\widetilde {I^n} = I^n$ for all $n \ge 1$. 
\end{Corollary}

It is also known that $\widetilde {I^n} = I^n$ for all $n \ge 1$ if and only if $G(I)$ contains a non-zerodivisor of positive degree \cite[(1.2)]{HLS}. This fact can be used to construct examples with $s^*(I) = 1$ and $\reg R(I)$ arbitrarily large. 

\begin{Example} \label{example}
{\rm Let $R = k[X]/P$, where $k[X]$ is a polynomial ring and $P$ is a homogeneous prime generated by forms of degree $d$.   Let $A$ be the localization of $R$ at its maximal graded ideal. Then $G(\mm) \cong R$. Since $\depth R > 0$, $s^*(\mm) = 1$.  By Theorem \ref{regularity}, $\reg R(\mm) = \reg G(\mm) = \reg R$. Since $\reg R+1$ is an upper bound for the degree of the generators of $P$ \cite{EG}, $\reg R(\mm) \ge d-1$.}
\end{Example}

Despite the possible large difference between $\reg R(I)$ and $s^*(I)$ we can characterize $\reg R(I)$ in terms of $s^*(I)$ in the following case. \par

Recall that $A$ is called a {\em Buchsbaum ring} if every system of parameters $x_1,...,x_r$ of $A$ is a {\em weak sequence}, i.e. 
$$(x_1,...,x_{i-1}): x_i = (x_1,...,x_{i-1}):\mm$$
for $i = 1,...,r$. \par

Huneke showed that $A$ is a Buchsbaum ring if and only if every system of parameters forms a $d$-sequence \cite[Proposition 1.7]{Hun}. Therefore, $\reg R(I) = 0$ and $s^*(I) = 1$ if $I$ is a parameter ideal in a Buchsbaum ring. 
If $I$ is not a parameter ideal, we have the following formula for $\reg R(I)$, which is not known even in the case $A$ is a regular local ring.

\begin{Theorem} \label{equality}
Let $A$ be a two-dimensional Buchsbaum ring with $\depth A > 0$. 
Let $I$ be an $\mm$-primary ideal, which is not a parameter ideal. 
Let $J$ be a minimal reduction of $I$. Then  
$$\reg R(I) = \max\{r_J(I), s^*(I)\}  = \min\{n \ge r_J(I)|\ \widetilde {I^n} = I^n\}.$$
\end{Theorem}

\begin{proof}
By Theorem \ref{regularity}, $\reg R(I) \ge r_J(I)$. 
Since $I$ is not a parameter ideal, $r_J(I) \ge 1$. Hence, $\reg R(I) \ge 1$.
By Proposition \ref{bound} (ii), this implies $\reg R(I) \ge s^*(I)$. 
Thus, $\reg R(I) \ge  \max\{r(I),s^*(I)\}.$
Since 
$$\max\{r_J(I),s^*(I)\} \ge \min\{n \ge r_J(I)|\ \widetilde {I^n} = I^n\},$$
it suffices to show that $\reg R(I) \le \min\{n \ge r_J(I)|\ \widetilde {I^n} = I^n\}.$\par

By Lemma \ref{generating}, there is a superficial sequence $x,y$ of $I$ such that $J = (x,y)$.
Since $\depth A > 0$, $I$ is a regular ideal. Hence $0:x = 0$.
By Theorem \ref{regularity}, this implies 
$$\reg R(I) = \min\{n \ge r_J(I)|\ I^{n+1} \cap [(x):y] = xI^n\}.$$
We will show that $I^{n+1} \cap [(x):y] = I^{n+1} \cap (x)$ for $n \ge r_J(I)$.
Let $f$ be an arbitrary element of $I^{n+1} \cap [(x):y]$. Since $I^{n+1} = (x,y)I^n$,
there are elements $g,h \in I^n$ such that $f = gx +hy$.
Since $fy \in (x)$, $h \in (x):y^2$.
Since $x,y$ is a $d$-sequence, $(x):y^2 = (x):y$.
Hence $hy \in y[(x):y^2] \subseteq y[(x):y] \subseteq (x)$.
This implies  $f  \in (x)$. So we can conclude that $I^{n+1} \cap [(x):y] \subseteq I^{n+1} \cap (x)$. 
Since the converse inclusion is obvious, $I^{n+1} \cap [(x):y] = I^{n+1} \cap (x)$. 
Therefore, 
\begin{align*}
\reg R(I) & = \min\{n \ge r_J(I)|\ I^{n+1} \cap (x) = xI^n\}\\
& = \min\{n \ge r_J(I)|\ I^{n+1}:x = I^n\}.
\end{align*}

We have $I^n \subseteq I^{n+1}: x \subseteq \widetilde {I^{n+1}}:x$.
By \cite[Lemma 3.1 (5)]{RV}, $\widetilde {I^{n+1}}:x = \widetilde {I^n}$.
Therefore,  $I^{n+1}:x = I^n$ if $\widetilde {I^n} = I^n$. So we can conclude that  
$$\reg R(I) \le \min\{n \ge r_J(I)|\ \widetilde {I^n} = I^n\}.$$ 
\end{proof}

The formula $\reg R(I) =  \min\{n \ge r_J(I)|\ \widetilde {I^n} = I^n\}$ provides an effective tool for the computation of $\reg R(I)$ because we only need  to check the condition $\widetilde {I^n} = I^n$ successively for $n \ge r_J(I)$. Moreover, comparing to Theorem \ref{regularity}, we do not need a superficial sequence which generates a reduction of $I$. To find such a sequence is not easy in general. 
 
\begin{Remark}
{\rm The proof of Theorem \ref{equality} leads us to raise the question of whether $\reg R(I) =  \min\{n \ge r_J(I)|\ I^{n+1}:I = I^n\}$. We could not find any counter-example to this formula. }
\end{Remark} 

Rossi and Swanson \cite[Section 4]{RS} asked whether $s^*(I) \le r_J(I)$ always holds. 
Under the assumption of Theorem \ref{regularity}, a positive answer would imply $r_J(I) = \reg R(I)$
for all minimal reductions $J$ of $I$.  
However, Huckaba \cite[Example 3.1]{Huc} showed that $r_J(I)$ depends on the choice of $J$.
So we obtain a negative answer to the above question.\par
 
On the other hand, we can give the following simple condition for the invariance of the reduction numbers.

\begin{Corollary} \label{invariance}
Let $A$ be a two-dimensional Buchsbaum ring with $\depth A > 0$ and $I$ an $\mm$-primary ideal.
If there exists a minimal reduction $J$ of $I$ such that $s^*(I) < r_J(I)$, 
then the reduction numbers of all minimal reductions of $I$ equal $\reg R(I)$.
\end{Corollary}

\begin{proof}
Since $r_J(I) > 1$, $I$ is not a parameter ideal.
By Theorem \ref{equality}, the assumption implies $\reg R(I) = r_J(I)$.
If there is a minimal reduction $J'$ of $I$ with $r_{J'}(I) \neq r_J(I)$,  
from the formula $\reg R(I) = \max\{r_{J'}(I),s^*(I)\}$ we can deduce that $\reg R(I) = s^*(I)$, a contradiction.
\end{proof}

One can easily find examples, where $r_J(I)$ is arbitrarily larger than $s^*(I)$.  
For instance, let $A$ be a  local ring as in Example \ref{example}. 
Then $s^*(\mm) = 1$ while $\reg R(\mm)$ can be arbitrarily large.
If we choose $A$ to be a two-dimensional Buchsbaum ring, then $r_J(\mm) = \reg R(\mm)$ by Theorem \ref{equality}.  

\begin{Remark}  \label{question}
{\rm Let $br(I)$ denote the big reduction number of $I$ which is defined by
$$br(I) := \max\{r_J(I)|\ J \text{ is a minimal reduction of } I\}.$$ 
If $br(I) \ge s^*(I)$, then $\reg R(I) = br(I)$ by Theorem \ref{equality}.
Under the assumption of Theorem \ref{equality}, we could not find any example with $br(I) < s^*(I)$. 
So we raise the question whether $\reg R(I) = br(I)$ always holds in this case.}
\end{Remark}

In the following we will give an alternative formula for $\reg R(I)$, which involves only the Ratliff-Rush closure of a power of $I$. This is based on the following observation.

\begin{Lemma} \label{reduction}
Let $A$ be a two-dimensional Buchsbaum ring with $\depth A > 0$ and $I$ an $\mm$-primary ideal.  
Let $x,y$ be a superficial sequence of $I$ such that $J = (x,y)$ is a minimal reduction of $I$.
Set $r = r_J(I)$. For $n \ge r$, we have
$$I^{n+1}: x = I^n + y^{n-r}(I^{r+1}:x).$$
\end{Lemma}

\begin{proof}
The case $n = r$ is trivial. For $n \ge r+1$, let $f$ be an arbitrary element of $I^{n+1}:x$. 
Since $I^{n+1} = (x,y)I^n$, there are elements $g,h \in I^n$ such that $xf = xg + yh$.
From this it follows that $h \in I^n \cap [(x):y]$.
As showed in the proof of Theorem \ref{equality}, 
$$I^n \cap [(x):y] = I^n \cap (x) = x(I^n:x).$$ 
Hence $h = xh'$ for some element $h' \in I^n:x$. Thus, $xf = xg + xyh'$. 
Since $x$ is a non-zerodivisor, $f = g +yh'  \in I^n + y(I^n:x)$.
So we have $I^{n+1}: x \subseteq I^n + y(I^n:x)$. 
Since the inverse inclusion is obvious,  we can conclude that 
$$I^{n+1}:x = I^n + y(I^n:x).$$
Applying this formula to $I^n:x$ and so on, we get
$I^{n+1}: x = I^n + y^{n-r}(I^{r+1}:x).$  
\end{proof}

\begin{Theorem} \label{alternative}
Let $A$ be a two-dimensional Buchsbaum ring with $\depth A > 0$ and $I$ an $\mm$-primary ideal.
Let $J$ be a minimal reduction of $I$. Set $r =  r_J(I)$. Then
$$\reg R(I) = \min\{n \ge r|\ \widetilde {I^r} = I^n:I^{n-r}\}.$$
\end{Theorem}

\begin{proof}
If $r = 0$, $I$ is a parameter ideal. Since $A$ is Buchsbaum, $I$ is generated by a $d$-sequence.
Hence $\reg R(I) = 0$ by \cite[Corollary 5.7]{Tr2}. 
In this case, the above formula is trivial. 
Therefore, we may assume that $I$ is not a parameter ideal. \par

Let $J = (x,y)$, where $x,y$ is a superficial sequence of $I$.
By the proof of Theorem \ref{equality}, we have
$$\reg R(I) = \min\{n \ge r|\ I^{n+1}: x = I^n\}.$$
By Lemma \ref{reduction},  
$$I^{n+1}: x = I^n + y^{n-r}(I^{r+1}:x) \subseteq I^n + I^{n-r}(\widetilde {I^{r+1}}:x).$$
By \cite[Lemma 3.1 (5)]{RV}, $\widetilde {I^{r+1}}:x = \widetilde {I^r}$.   
If $\widetilde {I^r} = I^n:I^{n-r}$, we have
$$I^n + I^{n-r}(\widetilde {I^{r+1}}:x) = I^n+ I^{n-r}(I^n:I^{n-r}) = I^n \subseteq I^{n+1}:x.$$
So we can conclude $I^{n+1}: x = I^n$.  Hence, 
$$\reg R(I) \ge \min\{n \ge r|\ \widetilde {I^r} = I^n:I^{n-r}\}.$$

To show the converse inequality we observe that $\widetilde{I^n}:x^{n-r} = \widetilde {I^r}$ \cite[Lemma 3.1 (5)]{RV}. 
Since $\widetilde {I^r} \subseteq \widetilde {I^n}: I^{n-r} \subseteq \widetilde {I^n}:x^{n-r}$, 
this implies  $ \widetilde {I^r} = \widetilde {I^n}: I^{n-r}$.
If $\widetilde {I^n} = I^n$, then $\widetilde {I^r} = I^n:I^{n-r}$. 
Therefore, using Theorem \ref{equality}, we have
$$\reg R(I)  = \min\{n \ge r|\ \widetilde {I^n} = I^n\} \le \min\{n \ge r|\ \widetilde {I^r} = I^n:I^{n-r}\}. $$
\end{proof}


\section{Regularity of the fiber ring}

Throughout this section let $A$ be a finitely generated standard graded algebra over a field $k$ with $\dim A > 0$. 
Let $I$ be an ideal generated by homogeneous elements of the same degree $d$, $d \ge 1$.
Under this assumption, the fiber ring $F(I)$ is isomorphic to the algebra $k[I_d]$ generated by the elements of the component $I_d$ of $I$.  \par

In the following {\em we will identify $F(I)$ with $k[I_d]$}.
Let $\nn$ denote the maximal graded ideal of $F(I)$. 
Since $F(I)$ is a standard graded algebra over $k$,
$F(I) \cong G(\nn)$. By Theorem \ref{regularity}, 
$$\reg F(I) = \reg G(\nn) = \reg R(\nn),$$
and we can use a minimal reduction of $\nn$ to compute $\reg F(I)$. 
\par

In general, there is a natural correspondence between minimal reductions of $I$ and $\nn$. 
In our setting, this correspondence can be described as follows.  

\begin{Lemma} \label{correspondence}  
Let $J$ be a homogeneous reduction of $I$.
Let $\q$ be the ideal generated by $J_d$ in $F(I)$. 
Then $\q$ is a reduction of $\nn$ with $r_\q(\nn) = r_J(I)$.
\end{Lemma}

\begin{proof} 
By our assumption,  the ideals $\nn$ and $I$ or 
 $\q$ and $J$ have the the same generating set.
From this it follows that $\nn^{n+1} = \q\nn^n$ if and only if  $I^{n+1} = JI^n$.
Hence the conclusion is immediate.
\end{proof}

\begin{Theorem} \label{fiber} 
Let $x_1,...,x_s$ be a superficial sequence of homogeneous elements of $I$ such that $J = (x_1,...,x_s)$ is a reduction of $I$. Then 
\begin{align*} & \reg F(I) =  \min \big\{t \ge r_J(I) \mid \\
&\big((I^{n+1}) \cap [(x_1,...,x_{i-1}):x_i]\big)_{(n+1)d} = \big((x_1,...,x_{i-1})I^n)_{(n+1)d} \text{ for } n \ge t, i = 1,...,s \big\}.
\end{align*}
\end{Theorem}

\begin{proof}
By Theorem \ref{regularity} we have
$$I^{n+1} \cap [(x_1,...,x_{i-1}) :x_i] = (x_1,...,x_{i-1})I^n$$
for all $n \ge R(I)$, $i = 1,...,s$. Therefore,
$$\big(I^{n+1} \cap [(x_1,...,x_{i-1}):x_i]\big)_{(n+1)d} = \big((x_1,...,x_{i-1})I^t\big)_{(n+1)d}$$
for all $n \ge \reg R(I)$. Since $\nn^{t+1} = \oplus_{n \ge t}(I^{n+1})_{(n+1)d}$, this implies
$$\nn^{t+1} \cap [(x_1,...,x_{i-1})F(I) :x_i] = (x_1,...,x_{i-1})\nn^t$$
for $t \ge \reg R(I)$, $i = 1,...,s$.
Note that $\reg R(I) \ge r_J(I) = r_\q(\nn)$ by Theorem \ref{regularity} and Lemma \ref{correspondence}.
Then the above equations implies that  $x_1^*,...,x_s^*$ form a filter-regular sequence of $F(I) = G(\nn)$ \cite[Theorem 4.7]{Tr2}.
Hence, $x_1,...,x_s$ is a superficial sequence of $\nn$ by Lemma \ref{superficial}.
Now, we can apply Theorem \ref{regularity} to the reduction $\q = (x_1,...,x_s)F(I)$ of $\nn$, and obtain
\begin{align*}
& \reg F(I) =  \\
& \;\;\;\;\; \min\big\{t \ge r_J(I)|\  \nn^{t+1} \cap [(x_1,...,x_{i-1})F(I) :x_i] = (x_1,...,x_{i-1})\nn^t, i = 1,...,s\big\}.
\end{align*} 
This formula can be rewritten as in the statement of the theorem. 
\end{proof}

The relations in Theorem \ref{fiber} are only the initial parts of the relations  in Theorem \ref{regularity}. 
As an immediate consequence, we get the inequality $\reg R(I)  \ge \reg F(I),$ which was already observed by Eisenbud and Ulrich in \cite[Section 1]{EU}. \par

Since $I$ is generated by elements of the same degree, $R(I)$ is a bigraded algebra over $k$.   Eisenbud and Ulrich used the bigraded minimal free resolution to define $\reg R(I)$. That definition coincides with the definition of $\reg R(I)$ by means of local cohomology modules \cite{Ro}. \par 

Let $\mm$ be the maximal graded ideal of $A$. 
Eisenbud and Ulrich conjectured that if  $I$ is an $\mm$-primary ideal,  then $\reg R(I) = \reg F(I)$ \cite[Conjecture 1.3]{EU}.  
This conjecture is not true if $A$ is a non-Buchsbaum ring. This follows from the following observation.
Note that $A$ is called a Buchsbaum ring if $A_\mm$ is a Buchsbaum ring. 

\begin{Proposition} \label{Buchsbaum}
Assume that $\reg R(Q) = \reg F(Q)$ for every parameter ideal $Q$ generated by forms of the same degree in graded quotient rings of $A$. Then $A$ is a Buchsbaum ring. 
\end{Proposition}

\begin{proof}
It is well known that every system of parameters is analytically independent.
From this it follows that $F(Q)$ is isomorphic to a polynomial ring over $k$.
Hence $\reg R(Q) = \reg F(Q) = 0$.  
By  \cite[Corollary 5.7]{Tr2}, this implies that $Q$ is generated by a $d$-sequence.
In particular, every system of parameters of $A$, which consists of forms of the same degree, is a $d$-sequence. \par
Let $x_1,...,x_s$ be a homogeneous system of parameters of degree 2 of $A$, $s = \dim A$.
Applying the above fact to the quotient ring $A/(x_1,...,x_i)$, $i < s$, we can deduce that every homogeneous system  of parameters  $x_1,...,x_i,y_1,...,y_{r-i}$ of $A$, where $y_1,...,y_{s-i}$ are linear forms, is a $d$-sequence.
By \cite[Corollary 2.6]{Tr}, a local ring $(B,\nn)$ is Buchsbaum if there exists a system of parameters $x_1',...,x_s'$ in $\nn^2$, $s = \dim B$, and a generating set $S$ for $\nn$ such that $x_1',...,x_i',y_1',...,y_{s-i}'$ is a $d$-sequence for every family $y_1',...,y_{s-i}'$ of $s-i$ elements of $S$, $i = 1,...,s$ (the term absolutely superficial sequence was used there for $d$-sequence). From this it follows that $A$ is a Buchsbaum ring. 
\end{proof}

The following example shows that $\reg R(I)$ can be arbitrarily larger than $\reg F(I)$ even when $I$ is a parameter ideal in an one-dimensional non-Buchsbaum ring.

\begin{Example} 
{\rm Let $A = k[x,y]/(x^t,xy^{t-1})$, $t \ge 2$. Then $A$ is a non-Buchsbaum ring for $t \ge 3$. Let $I = yA$. It is clear that $\reg F(I) = 0$. Using Theorem \ref{regularity}, we have 
$\reg R(I) = \min\{n \ge 0|\ y^{n+1}A \cap (0:yA) = 0\} = t-2.$} 
\end{Example}

We can prove the conjecture of Eisenbud and Ulrich in the following case.

\begin{Theorem} \label{Buch}
Let $A$ be a Buchsbaum ring with $\dim A \ge 1$,
Let $I$ be an $\mm$-primary ideal such that $\depth G(I) \ge \dim A-1$. 
Then 
$$\reg R(I) = \reg F(I) = r_J(I)$$
for every minimal reduction $J$ of $I$.
\end{Theorem}

\begin{proof}
By \cite[Theorem 1.2]{Tr1}, the assumption implies that $\reg R(I) = r_J(I)$. 
By Theorem \ref{fiber},  $r_J(I) \le \reg F(I) \le \reg R(I).$
So we must  have $\reg R(I) = \reg F(I) = r_J(I)$.
\end{proof}

The condition  $\depth G(I) \ge \dim A-1$ is satisfied if $\dim A = 1$. Therefore, we have the following consequence.

\begin{Corollary}
Let $A$ be an one-dimensional Buchsbaum ring.
Let $I$ be an $\mm$-primary ideal.
Then $\reg R(I) = \reg F(I)$.
\end{Corollary}

In the following we will show that there is a strong relationship between $\reg F(I)$ and the initial behavior of the ideals $\widetilde {I^n}$. \par

Let $s^*_\ini(I)$ denote the least integer $m$ such that  $(\widetilde {I^n})_{nd} = (I^n)_{nd}$ for all $n \ge m$.
Note that $nd$ is the initial degree of $\widetilde {I^n}$ and $I^n$.  

\begin{Lemma} \label{initial}
Let $I$ be a regular ideal. Then $s^*_\ini(I) \le \max\{\reg F(I), 1\}.$
\end{Lemma}

\begin{proof}
We have to show that $(\widetilde {I^n})_{nd} = (I^n)_{nd}$ for $n \ge \reg F(I)$.
Since $\widetilde {I^n} = \cup_{t \ge 0}I^{n+t}:I^t$, it suffices to show that
$(I^{n+t}:I^t)_{nd} \subseteq (I^n)_{nd}$ for $n \ge \reg F(I)$. \par

Let $x \in I_d$ be a superficial element of $\nn$. Since $\nn$ is a regular ideal, $x$ is a non-zerodivisor.
By Corollary \ref{intersection}, $\nn^{n+1} \cap (x) = x\nn^n$ for $n \ge \reg F(I)$. 
This implies $\nn^{n+1}: x = \nn^n$. 
From this it follows that $\nn^{n+t}: x^t = \nn^n$ for $n \ge \reg F(I)$.
Note that $(\nn^n)_d = (I^n)_{nd}$. Then 
$$(I^{n+t}:I^t)_{nd} \subseteq (I^{n+t}:x^t)_{nd} = (\nn^{n+t}:x^t)_d = (\nn^n)_d = (I^n)_{nd}
$$
for $n \ge \reg F(I)$, which implies the conclusion.
\end{proof}

\begin{Theorem} \label{equality 2} 
Let $A$ be a two-dimensional Buchsbaum ring with $\depth A > 0$.  
Let $I$ be an $\mm$-primary ideal, which is not a parameter ideal. 
Let $J$ be a homogeneous minimal reduction of $I$. Then 
$$\reg F(I) = \max\{r_J(I), s_\ini^*(I)\} = \min\{n \ge r_J(I)|\ (\widetilde {I^n})_{nd} = (I^n)_{nd}\}.$$
\end{Theorem}

\begin{proof}
By Theorem \ref{fiber} we have $\reg F(I) \ge r_J(I)$. Since $I$ is not a parameter ideal, $\nn$ is not generated by two elements. From this it follows that the defining equations of $F(I)$ have degree $> 1$. Hence $\reg F(I) > 0$ \cite{EG}. By Lemma \ref{initial}, this implies $\reg F(I) \ge s^*_\ini(I)$.
Thus,  
$$\reg F(I) \ge \max\{r_J(I), s_\ini^*(I)\} \ge \min\{n \ge r_J(I)|\ (\widetilde {I^n})_{nd} = (I^n)_{nd}\}.$$ 
Therefore,
it suffices to show that $\reg  F(I) \le \min\{n \ge r_J(I)|\  (\widetilde {I^n})_{nd} = (I^n)_{nd}\}.$ \par

Since $I = (I_d)$ is a regular ideal,  $I_d$ contains a non-zerodivisor. From this it follows that
$\depth F(I) = \depth k[I_d] > 0$. Let $\nn$ be the maximal graded ideal of $F(I)$ and $\q$ the ideal generated by $J_d$ in $F(I)$.
Then $\q$ is a minimal reduction of $\nn$ with $r_\q(\nn) = r_J(I)$ by Lemma \ref{correspondence}.
By Lemma \ref{generating} we may assume that $\q = (x,y)$, where $x,y$ is a superficial sequence for $\nn$.
By the proof of Theorem \ref{regularity} to $\nn$ we have  
$$\reg F(I) =   \min\{n \ge r_J(I)|\  \nn^{n+1} : x = \nn^n\}.$$
Thus, we only need to show that if $(\widetilde {I^n})_{nd} = (I^n)_{nd}$ for $n \ge r_J(I)$, then $\nn^{n+1} : x = \nn^n$.
Note that $\nn^{n+1} : x = \oplus_{t \ge n} (I^{t+1}:x)_{td}$ and $\nn^n = \oplus_{t \ge n}(I^t)_{td}$.
\par

Without restriction we may also assume that $x,y$ is a superficial sequence for $I$ and $J = (x,y)$. 
By Lemma \ref{reduction}, we have 
$I^{t+1}:x = I^t + y^{t-n}(I^{n+1}:x)$
for $t \ge n \ge r_J(I)$. 
Since $I^{n+1}:x \subseteq \widetilde {I^{n+1}}:x = \widetilde {I^n}$ \cite[Lemma 3.1(5)]{RV},  
$I^{t+1}:x \subseteq I^t + y^{t-n}\widetilde {I^n}.$
If $(\widetilde {I^n})_{nd} = (I^n)_{nd}$, then
$$(I^{t+1}:x)_{td} \subseteq  (I^t)_{td} + y^{t-n}(\widetilde {I^n})_{nd} \subseteq  (I^t)_{td} + y^{t-n}(I^n)_{nd}\\
= (I^t)_{td} \subseteq (I^{t+1}:x)_{td}. $$
From this it follows that $(I^{t+1}:x)_{td} = (I^t)_{td}$ for all $t > n$. Hence, $\nn^{n+1} : x = \nn^n$, as required.
\end{proof} 

\begin{Corollary} \label{initial 1}
Let $A$ be a two-dimensional Buchsbaum ring with $\depth A > 0$.
Let $I$ be an $\mm$-primary ideal. 
Let $J$ be a homogeneous minimal reduction of $I$.
Then $\reg R(I) = \reg F(I)$ if and only if $\widetilde {I^n} = I^n$ for the least integer $n \ge r_J(I)$ such that 
$(\widetilde {I^n})_{nd} = (I^n)_{nd}$.
\end{Corollary}

\begin{proof}
If $I$ is a parameter ideal, we have $\reg R(I) = \reg F(I) = 0$ by \cite[Corollary 5.7]{Tr2} and $\widetilde {I^n} = I^n$ for $n \ge 1$ by Corollary \ref{d-sequence}. If $I$ is not a parameter ideal, the conclusion follows from Theorem \ref{equality} and Theorem \ref{equality 2}. 
\end{proof}


\section{Monomial ideals in two variables}

In this section we will use the relationship between Castelnuovo-Mumford regularity and Ratliff-Rush closure to investigate the conjecture of Eisenbud and Ulrich for monomial ideals in two variables.\par

Let $A = k[x,y]$ be a polynomial ring over a field $k$ and $\mm = (x,y)$.
Let $I$ be an $\mm$-primary ideal generated by monomials of degree $d$, $d \ge 1$. 
Then $I$ contains $x^d,y^d$, and $J = (x^d,y^d)$ is a minimal reduction of $I$.  
It is well-known \cite{Sa} and easy to see that
$$\widetilde {I^n} = \bigcup_{t \ge 0}I^{n+t}: (x^{td},y^{td}).$$

\begin{Lemma}  \label{middle 1}
Let $I= (x^d,y^d) + (x^{d-i}y^i\mid a \le i \le b)$, where $a \le b < d$ are given positive integers.  
Then $\widetilde {I^n} = I^n$ for all $n \ge 1$. 
\end{Lemma}

\begin{proof}
Let $x^iy^j$ be an arbitrary monomial of $\widetilde {I^n}$. 
Then $x^{i+td}y^j \in I^{n+t}$ for some $t \ge 1$. 
Since $I^{n+t}$ is generated by monomials of degree $(n+t)d$, 
$x^{i+td}y^j$ is divisible by a monomial $x^{(n+t)d-c}y^c \in I^{n+t}$. 
The divisibility implies $i+td \ge (n+t)d-c$ and $j \ge c$.  \par

If $j < na$, then 
$$(n+t)d - c \ge (n+t)d - j > (n+t)d + na = td + n(d-a).$$
Let $M = \{x^d,x^{d-a}y^a,x^{d-a-1}y^{a+1},...,x^{d-b}y^b,y^d\}$ be the set of the monomial generators of $I$. 
Then $x^{(n+t)d-c}y^c$ is a product of $n+t$ monomials of $M$. Let $s$ be the number of copies of $x^d$ among these $n+t$ monomials of $M$. If $s < t$, we would have $(n+d)d -c \le sd + (n+t-s)(d-a)$ because the exponent of $x$ in each monomial in $M \setminus \{x^d\}$ is less or equal $d-a$. Since 
$$sd + (n+t-s)(d-a) =  td + n(d-a) - (t-s)a < td + n(d-a),$$
we would get $(n+d)d -c < td+n(d-a)$, a contradiction. Therefore, we must have $s \ge t$. From this it follows that $x^{nd-c}y^c = x^{(n+t)d-c}y^c/x^{td}$ is a product of $n$ monomials in $M$. Hence $x^{nd-c}y^c \in I^n$. Since $x^iy^j$ is divisible  by $x^{nd-c}y^c$, $x^iy^j \in I^n$. \par
 
By symmetry, if $i < n(d-b)$, we can also show that $x^iy^j \in I^n$. \par
  
Now, we may assume that $i \ge n(d-b)$ and $j \ge na$. 
Let $Q$ denote the ideal generated by the monomials $x^{d-j}y^j$, $a \le j \le b$. 
It is clear that $Q^n$ is generated by the monomials $x^{nd-j}y^j $, $na \le j \le nb$.
If $ j  <  nb$, $x^iy^j$ is divisible by $x^{nd-j}y^j$ because $i \ge nd - c \ge nd-j$.
Therefore, $x^iy^j \in Q^n$. Since $Q \subset I$, we obtain $x^iy^j \in I^n$. 
If $j \ge nb$, then $x^iy^j$ is divisible by $x^{n(d-b)}y^{nb} = (x^{d-b}y^b)^n \in I^n$.
Thus, we always have $x^iy^j \in I^n$.
Therefore, we can conclude that $\widetilde {I^n} = I^n$.
\end{proof}

Lemma \ref{middle 1} gives a large class of monomial ideals in two variables for which the conjecture of Eisenbud and Ulrich holds.

\begin{Theorem}  \label{middle 2}
Let $I= (x^d,y^d) + (x^{d-i}y^i\mid a \le i \le b)$, where $a \le b < d$ are given positive integers.  Then
$$\reg R(I) = \reg F(I) = r_J(I)$$
for any homogeneous minimal reduction $J$ of $I$.
\end{Theorem}

\begin{proof}
By  Lemma \ref{middle 1}, we have $s^*(I) = 1$. Since $s^*(I) \ge s^*_\ini(I) \ge 1$, we also have $s^*_\ini(I) = 1$.
Applying Theorem \ref{equality} and Theorem \ref{equality 2}, we obtain $\reg R(I) = \reg F(I) = r_J(I)$.
\end{proof}

Theorem \ref{middle 2} contains the case $I$ is an $\mm$-primary ideal generated by three monomials.

\begin{Corollary}  
Assume that $I = (x^d,x^{d-a}y^a,y^d)$, $1 \le a <d$. Then
$$\reg R(I) = \reg F(I) = d/(a,d)-1.$$
\end{Corollary}

\begin{proof}
This is the case $a  = b$ of Theorem \ref{middle 2}.
Therefore, $\reg R(I) = \reg F(I) = r_J(I)$. For $J = (x^d,y^d)$, 
it is easy to check that $r_J(I) = d/(a,d)-1$.
\end{proof}

Now we will present another large class of monomial ideals in two variables for which the conjecture of Eisenbud and Ulrich holds.  

\begin{Theorem} \label{neighbor}
Let $I$ be an ideal in $k[x,y]$ which is generated by monomials of degree $d \ge 2$. Assume that $x^d,x^{d-1}y,y^d \in I$. Then
$\reg R(I) = \reg F(I).$
\end{Theorem}

\begin{proof}
Let $n$ be the least integer $n \ge r_J(I)$ such that $(\widetilde {I^n})_{nd} = (I^n)_{nd}$. 
By Corollary \ref{initial 1}, we only need to show that $\widetilde {I^n} = (I^n)$.  \par

Let $x^iy^j$ be an arbitrary monomial of $\widetilde {I^n}$. 
Then $x^iy^{j+td} \in I^{n+t}$ for some $t \ge 1$. 
Since $I^{n+t}$ is generated by monomials of degree $(n+t)d$, there exists a monomial $x^ay^{(n+t)d-a} \in I^{n+t}$ such that $x^iy^{j+td}$ is divisible by $x^ay^{(n+t)d-a}$. 
By the divisibility, we have $i \ge a$ and $j + td \ge (n+t)d-a$.  \par

If $i \ge nd$, then $x^iy^j$ is divisible by $x^{nd} \in I^n$. \par
If $i < nd$,  then $a < nd$. We have 
\begin{align*}
(x^ay^{nd-a})x^{(nd-a)d} & = (x^{d-1}y)^{nd-a}x^{nd} \in I^{nd-a+n}\\
(x^ay^{nd-a})y^{td} & = x^ay^{(n+t)d-a} \in I^{n+t}.
\end{align*}
Set $s = \max\{(nd-a),t\}$.  
Then $(x^ay^{nd-a})x^{sd}, (x^ay^{nd-a})y^{sd} \in I^{n+s}$.
Hence $x^ay^{nd-a} \in I^{n+s}:(x^{sd},y^{sd}) \subseteq \widetilde {I^n}$.
Since $(\widetilde {I^n})_{nd} \in (I^n)_{nd}$, $x^ay^{nd-a} \in I^n$.
Since $i \ge a$ and $j \ge nd-a$, $x^iy^j$ is divisible by $x^ay^{nd-a}$. Thus, $x^iy^j \in I^n$. \par
So we can conclude that $\widetilde {I^n} = (I^n)$.
\end{proof}

Actually, the proof of the above cases is obtained by translating everything in terms of lattice points  in $\NN^2$. 
For every set $Q \subseteq A$ we consider the set $\Exp(Q)$ of the exponent vectors of the monomials in $Q$. \par

Let $E = \Exp(I_d)$. Then $\Exp((I^n)_{nd}) = nE$ for all $n \ge 0$,  
where $nE$ is the Minkowski sum of $n$ copies of $E$. Let   
$$F_n = \{\a \in \NN^2 \mid \a + t\e_1, \a+t\e_2 \in (n+t)E  \text{ for some } t \ge 0\},$$
where $\e_1 = (d,0)$ and $\e_2 = (0,d)$. Then
$F_n = \Exp((\widetilde {I^n})_{nd})$.
Hence $(\widetilde {I^n})_{nd} = (I^n)_{nd}$ if and only if $F_n = nE$.
Similarly, if we let 
$$F^*_n = \{\a \in \NN^2\mid \a + t\e_1, \a+t\e_2 \in (n+t)E + \NN^2 \text{ for some } t \ge 0\},$$ 
then $\Exp(\widetilde {I^n}) = F^*_n$. Hence $\widetilde {I^n} = I^n$ if and only if $F^*_n = nE+\NN^2$. 
    
\begin{Remark}
{\rm   
Let $S \subseteq \NN^2$ denote the additive monoid generated by $E$. Let $k[S]$ denote the semigroup ring $k[S]$ of $S$ over $k$. Then $k[S] \cong F(I)$. Let $S^* = \cup_{n \ge 0}F^*_n.$
Then $S^*$ is also an additive monoid. By \cite[Lemma 1.1]{Tr'}, $k[S]$ is Cohen-Macaulay or Buchsbaum if and only if  $S^* = S$ or $S^*+(S \setminus \{0\}) \subseteq  S$. In particular,  
$$H_\nn^1(k[S]) \cong k[S^*]/k[S],$$
where $\nn$ denote the maximal graded ideal of $k[S]$.   
}
\end{Remark}

Now we will give an example such that $\reg R(I) = s^*(I) = r_J(I)$ for $J = (x^d,y^d)$, but $s^*(I) > r_{J'}(I)$ for another minimal reduction $J'$ of $I$. This example shows that $r_J(I)$ is not always the minimal reduction number of $I$ for $J = (x^d,y^d)$.

\begin{Example} \label{Huckaba}
{\rm
Let $I = (x^7,x^6y,x^2y^5,y^7)$. First, we will show that $\reg R(I) = r_J(I)$ for $J = (x^7,y^7)$.  By Theorem \ref{neighbor} we know that $\reg R(I) = \reg F(I)$. Since $\reg F(I) \ge r_J(I)$ by (2), it suffices to show that $\reg F(I) = r_J(I)$. It is easy to check that $r_J(I) = 4$ and $(\widetilde {I^4})_{28} = (I^4)_{28}$. By Theorem \ref{equality 2}, this implies $\reg F(I) = 4$. On the other hand, it is shown in \cite[Example 3.1]{Huc} that $r_{J'}(I) \le 3$ for $J' = (x^7,x^6y+y^7)$. By Theorem \ref{equality}, this implies $\reg R(I) = s^*(I) = r_J(I)$. 
}
\end{Example}

The following example shows that we may have $\reg R(I) = s^*(I) > r_J(I)$, where $J = (x^d,y^d)$.  
We do not know whether $br(I) = r_J(I)$ in this example. 
If $br(I) = r_J(I)$, this will give a negative answer to 
the question of Remark \ref{question}.

\begin{Example} \label{big}
{\rm Let $I = (x^{17-i}y^i|\ i = 0,1,3,5,13,14,16,17)$. By \cite[Example 3.2]{HHS}, we have $\reg F(I) = 4 > r_J(I) = 3$, where $J = (x^{17},y^{17})$. By Theorem \ref{neighbor} we have $\reg R(I) = \reg F(I) = 4$. Hence,
Theorem \ref{equality} implies $\reg R(I) = s^*(I) > r_J(I)$.
}
\end{Example}



\begin{thebibliography}{1} 


\bibitem{BM}
D. Bayer and D. Mumford,
What can be computed in algebraic geometry?  In: Computational Algebraic Geometry and Commutative Algebra (Cortona 1991), 1--48, Cambridge
University Press, 1993.

\bibitem{CZ}
B.~T. Cortadellas and S. Zarzuela,
On the structure of the fiber cone of ideals with analytic spread one,
J. Algebra 317 (2007), no. 2, 759--785. 

\bibitem{Cr} V. Crispin Quinonez,  Ratliff-Rush monomial ideals, in: Algebraic and Geometric Combinatorics, Contemp. Math. {\bf 423} (2007), 43--50. 

\bibitem{CHT} S.D. Cutkosky, J. Herzog and N.V. Trung, 
Asymptotic behaviour of the Castelnuovo-Mumford regularity,
Compositio Math. 118 (1999), no. 3, 243--261.

\bibitem{DGV}
L.~R. Doering, T. Gunston, W. Vasconcelos, 
Cohomological degrees and Hilbert functions of graded modules, 
Amer. J. Math. 120  (1998), 493--504.

\bibitem{DH}
L.~X. Dung and L.~T. Hoa, 
Castelnuovo-Mumford regularity of associated graded modules and fiber cones of filtered modules, 
Comm. in Algebra. 40 (2012), 404--422.

\bibitem{EG}
D.~Eisenbud and S.~Goto,
Linear free resolutions and minimal multiplicities, 
J. Algebra 88 (1984),  89--133.

\bibitem{EU}
D.~Eisenbud and B. Ulrich,
Notes on regularity stabilization, Proc. Amer. Math. Soc. 140 (2012), no. 4, 1221--1232.

\bibitem{El}
J. Elias, On the computation of Ratliff-Rush closure, 
J. Symbolic Comput. 37 (2004), 717--725.

\bibitem{JN}
 A. V. Jayanthan and R. Nanduri,
Castelnuovo-Mumford Regularity and Gorensteiness of Fiber Cone, 
Comm. in Algebra 40 (2012), 1338--1351.

\bibitem{HJLS}
W. Heinzer, B. Johnston, D. Lantz and K. Shah, 
Coefficient ideals in and blowups of a commutative Noetherian domain,
J. Algebra 162 (1993), 355--391.

\bibitem{HLS}
W. Heinzer, D. Lantz and K. Shah, 
The Ratliff-Rush ideals in a Noetherian ring, 
Comm. in Algebra 20 (1992), 591--622.

\bibitem{HHS}
M. Hellus, L.T. Hoa and J. St\"uckrad,
Castelnuovo-Mumford regularity and the reduction number of some monomial curves,
Proc. Amer. Math. Soc. 318 (2010), 27--35.

\bibitem{HPT} J. Herzog, D. Popescu and N.V. Trung,
Regularity of Rees algebras,
J. London Math. Soc. (2) 65 (2002), 320--338.

\bibitem{Ho1}
L.T. Hoa, 
Reduction numbers of equimultiple ideals, 
J. Pure Appl. Algebra 109 (1996), 111--126.

\bibitem{Ho}
L.T. Hoa,
A note on the Hilbert-Samuel Hilbert function in a two-dimensional local ring,
Acta Math. Vietnamica 21 (1996), 335--347.

\bibitem{Huc}
S. Huckaba, 
Reduction numbers for ideals of higher analytic spread, 
Math. Proc. Cambridge Phil. Soc. 102 (1987), 49--57.

\bibitem{Hun}
C. Huneke,
The theory of $d$-sequences and powers of ideals,
Advances in Math. 46 (1982), 249--279.

\bibitem{Ko} V. Kodiyalam, 
Asymptotic behaviour of Castelnuovo-Mumford regularity, 
Proceedings of Amer. Math. Soc. 128, no. 2, (1999), 407--411.

\bibitem{Li1}
C.~H. Linh, 
Upper bound for the Castelnuovo-Mumford regularity of associated graded modules,
Comm. in Algebra 33 ( 2005), 1817--1831.

\bibitem{Li2}
C. H. Linh,
Castelnuovo-Mumford regularity and degree of nilpotency, 
Math. Proc. Cambridge Philos. Soc. 142 (2007), 429--437.

\bibitem{Pu}
T. Puthenpurakal,
Ratliff-Rush filtration, regularity and depth of higher associated graded modules,
J. Pure Appl. Algebra 208 (2007), No. 1, 159--176.

\bibitem{RR} 
L. J. Ratliff, Jr. and D. E. Rush, 
Two notes on reductions of ideals, 
Indiana Univ. Math. J. 27 (1978), 929--934.

\bibitem{Ro} T. R\"omer,
Homological properties of bigraded algebras, Illinois J. Math. 45 (2001), 1361--1376.

\bibitem{RS}
M. Rossi and I. Swanson,
Notes on the behavior of the Ratliff-Rush filtration, in: Commutative Algebra (Grenoble/Lyon, 2001),
Contemp. Math. 331, 313--328, Amer. Math. Soc., 2003.

\bibitem{RV}
M. Rossi and G. Valla,
Hilbert functions of filtered modules,
Lect. Notes of Unione Matematica. Italiana vol. 9, Springer, 2010.

\bibitem{RTV1}
M. E. Rossi, N. V. Trung and G. Valla, 
Castelnuovo-Mumford regularity and extended degree,
Trans. Amer. Math. Soc. 355 (2003), 1773-1786.

\bibitem{Sa}
J. Sally, 
Ideals whose Hilbert function and Hilbert polynomial agree at $n = 1$,
J. Algebra 157 (1993), 534--547.

\bibitem{St}
B.~Strunk,
Castelnuovo-Mumford regularity, postulation numbers, and reduction numbers,
J. Algebra 311 (2007), 538--550.

\bibitem{Tr}
N.~V. Trung,
Absolutely superficial sequences,
Math. Proc. Cambridge Philos. Soc. 93 (1983), 35--47.

\bibitem{Tr'}
N.~V. Trung,
Projections of one-dimensional Veronese varieties,
Math. Nachr. 118  (1984), 47-67.

\bibitem{Tr1}
N.~V. Trung,
Reduction exponent and degree bound for the defining equations of  graded rings,
Proc. Amer. Math. Soc. 101 (1987), 229--236.

\bibitem{Tr2}
N.~V. Trung,
The Castelnuovo regularity of the Rees algebra and the  associated graded ring,
Trans. Amer. Math. Soc. 35 (1998), 2813--2832.

\bibitem{Tr3}
N.~V. Trung,
Castelnuovo-Mumford regularity and related invariants, in: Commutative Algebra,  Lecture Notes Series 4, Ramanujan Mathematical Society, 2007, 157--180.

\bibitem{TW} 
N.~V. Trung and H-S. Wang, 
On the asymptotic linearity of Castelnuovo-Mumford regularity, 
J. Pure Appl. Algebra 201 (2005), 42--48.

\bibitem{Va}
W. Vasconcelos,
Integral Closure: Rees algebras, Multiplicities, and Algorithms, Springer, 2005.

\end{thebibliography}
\end{document}